\documentclass[12pt,leqno,a4paper]{amsart}
\usepackage{amssymb,enumerate}
\newcommand{\CC}{{\mathbb{C}}}
\newcommand{\FF}{{\mathbb{F}}}

\newcommand{\fA}{{\mathfrak{A}}}
\newcommand{\fS}{{\mathfrak{S}}}

\newcommand{\bG}{{\mathbf{G}}}
\newcommand{\bL}{{\mathbf{L}}}
\newcommand{\bT}{{\mathbf{T}}}

\newcommand{\cE}{{\mathcal{E}}}

\newcommand{\Irr}{{\operatorname{Irr}}}
\newcommand{\Syl}{{\operatorname{Syl}}}
\newcommand{\SC}{{\operatorname{sc}}}

\newcommand{\GL}{{\operatorname{GL}}}
\newcommand{\PGL}{{\operatorname{PGL}}}
\newcommand{\PSL}{{\operatorname{L}}}
\newcommand{\SL}{{\operatorname{SL}}}
\newcommand{\SU}{{\operatorname{SU}}}

\newcommand{\tw}[1]{{}^#1\!}
\newcommand{\Ph}[1]{\Phi_#1}

\let\la=\lambda

\newtheorem{thm}{Theorem}[section]
\newtheorem{lem}[thm]{Lemma}

\newtheorem{cor}[thm]{Corollary}
\newtheorem{prop}[thm]{Proposition}

\newtheorem*{conjA}{Conjecture A}
\newtheorem*{thmB}{Theorem B}
\newtheorem*{conjC}{Conjecture C}

\theoremstyle{definition}

\theoremstyle{remark}

\begin{document}

\title[The Projective Height Zero Conjecture]{The Projective Height Zero Conjecture}

\date{\today}

\author{Gunter Malle}
\address{FB Mathematik, TU Kaiserslautern, Postfach 3049,
         67653 Kaisers\-lautern, Germany.}
\email{malle@mathematik.uni-kl.de}
\author{Gabriel Navarro}
\address{Department of Mathematics, Universitat de Val\`encia, 46100 Burjassot,
        Spain.}
\email{gabriel@uv.es}

\thanks{The first author gratefully acknowledges financial support
by SFB TRR 195. The second author is partially supported by the
Spanish Ministerio de Educaci\'on y Ciencia Proyectos  MTM2016-76196-P  and
Prometeo II/Generalitat Valenciana.}



\begin{abstract}
We propose a projective version of the celebrated Brauer's Height Zero
Conjecture on characters of finite groups and prove it, among other cases,
for $p$-solvable groups as well as for (some) quasi-simple groups.
\end{abstract}

\maketitle


\section{Introduction} \label{sec:intro}

Let $G$ be a finite group and let $p$ be a prime. Recent investigations of
N.~Rizo and the second author \cite{NR} on blocks relative to characters of
normal subgroups lead us to
propose Conjecture~A below. Recall that if $B$ is a Brauer $p$-block of $G$,
we denote by $\Irr(B)$ the irreducible complex characters in $B$.
If $Z$ is a normal subgroup of $G$ and $\lambda\in\Irr(Z)$, then
$\Irr(B|\lambda)$ is the set of $\chi \in \Irr(B)$ such that
$\lambda$ is a constituent of the restriction $\chi_Z$. Furthermore,
$\Irr_0(B|\la)$ denotes the set of characters in $\Irr(B|\lambda)$ of height~0.

\begin{conjA}   \label{conj:projBHZ}
 Let $G$ be a finite group, $p$ a prime, and $B$ be a $p$-block of $G$ with
 defect group $D$. Suppose that $Z\le G$ is a central $p$-subgroup of $G$, and
 let $\la\in\Irr(Z)$. Then:
 $$\Irr(B|\la)=\Irr_0(B|\la)\qquad\Longleftrightarrow\qquad
   \text{$D/Z$ is abelian and $\la$ extends to $D$}.$$
\end{conjA}

Of course, when $Z=1$ then Conjecture~A is Brauer's Height Zero Conjecture.
In fact, as we shall show, the ``\emph{if}'' direction of Conjecture A follows
from the Height Zero Conjecture, and therefore  it is true by the main result
of Kessar--Malle \cite{KM13}. (Indeed, if $\la\in\Irr(Z)$ is faithful then the
condition that $D/Z$ is abelian and $\la$ extends to $D$ is equivalent to
$D$ being abelian.) The ``\emph{only if}'' direction, however, does
not seem to be implied by neither the Height Zero Conjecture nor the inductive
Alperin--McKay condition (which as we know now implies the Height Zero
Conjecture \cite{NS,KM17}). It is a (non-trivial) theorem of M.~Murai \cite{MM}
that if some $\chi \in \Irr(B|\lambda)$ has height zero, then $\lambda$
extends to $D$. Hence, the essential new part of our proposed conjecture is
that if all $\Irr(B|\lambda)$ have height zero, then $D/Z$ is abelian.

C.~Eaton has asked us if Conjecture A could still hold only assuming that
$\lambda$ is a $G$-invariant character of a normal $p$-subgroup $Z$ of $G$
and that all characters in $\Irr(B|\theta)$ have the same height as $\theta$,
as happens with Dade's Projective Conjecture. As a matter a fact, we believe
that this is the case for the ``\emph{only if}'' direction, but not for
the ``\emph{if}'' direction.

Our main results on Conjecture~A can be summarised as follows:
\newpage
\begin{thmB}
 \vbox{\  }
 \begin{enumerate}
  \item[\rm(a)] The ``\emph{if}" direction of Conjecture A is true.
  \item[\rm(b)] Conjecture~A is true for $p$-solvable groups as well as for
   nilpotent blocks.
  \item[\rm(c)] Quasi-simple groups do not provide minimal counterexamples to
   Conjecture A.
 \end{enumerate}
\end{thmB}

After studying Conjecture~A, it is tempting to conjecture that all characters
in $\Irr(D|\lambda)$ have the same degree implies that all characters in
$\Irr(B|\lambda)$ have the same height (and vice versa). However, the principal
2-block of the double cover $G$ of the alternating group $\fA_8$ provides an
interesting counterexample:
the non-trivial irreducible character $\lambda$ of $Z(G)$ is fully ramified
over $P \in \Syl_2(G)$, that is $\lambda^P =8\theta$ for some irreducible
$\theta \in \Irr(P)$, while there are characters in $B$ over $\lambda$ of
degrees 8, 24 and~48. In all the cases that we have checked, however, the
minimum non-zero height of the characters in $\Irr(B|\lambda)$ does coincide
with the minimum non-linear degree of the characters in $\Irr(D|\lambda)$.
It seems interesting to discuss if the Eaton--Moret\'o conjecture might have
a projective version.
\medskip

Our paper is built up as follows. In Section 2, we prove certain cases of
Conjecture~A. In Section 3, we complete the proof of Theorem~B by considering
quasi-simple groups.

Britta Sp\"ath has raised the possibility that it might
be possible to reduce Conjecture~A to the inductive Alperin--McKay conjecture
and a checking of Conjecture~A for quasi-simple groups only; this remains
open for the time being.

\medskip
\noindent{\bf Acknowledgement:} We would like to thank Radha Kessar for
asking about nilpotent blocks, and Charles Eaton for asking about the case
of arbitrary normal $p$-subgroups.

\section{Certain cases where Conjecture A holds}

\begin{thm}   \label{thm:p-solvable}
 Let $G$ be a finite group and let $p$ be a prime.
 Let $Z$ be a central $p$-subgroup of $G$ and $\lambda \in \Irr(Z)$.
 Assume that $B$ is a $p$-block of $G$ with defect group $D$.
 \begin{enumerate}
  \item[\rm(a)] If $(G,B)$ is a counterexample minimising $|G|$, to any of the
   two directions of Conjecture A, then $\lambda$ is faithful.
  \item[\rm(b)] The ``\emph{if}" direction of Conjecture A is true.
  \item[\rm(c)] If $D \in \Syl_p(G)$ then Conjecture~A holds if Brauer's
   Height Zero Conjecture holds for $G/Z$.
  \item[\rm(d)] Conjecture A is true for $p$-solvable groups.
  \item[\rm(e)] Conjecture A is true for nilpotent blocks.
 \end{enumerate}
\end{thm}

\begin{proof}
First of all, recall that if $\Irr(B|\lambda)$ contains a height zero
character, then $\lambda$ extends to $D$ by \cite[Thm.~4.4]{MM}. Also notice
that $\Irr(B|\lambda)$ is not empty, because $B$ covers the principal
block of $Z$ (see \cite[Thm 9.2]{Na}).
Since $K=\ker(\lambda)$ is a central $p$-subgroup of $G$, by
\cite[Thm.~9.10]{Na}, there is a unique block $\bar B$ of $G/K$ that is
contained in $B$. Furthermore, $D/K$ is a defect group of $\bar B$.
Hence the height of $\chi\in\Irr(\bar B)$ as a character of $G/K$ is the
same as its height as a character in $\Irr(B)$. Part (a) now easily follows.

To show~(b) assume now that $D/Z$ is abelian and that $\lambda$ extends to some
$\nu \in\Irr(D)$. In particular, $D' \le \ker(\nu)$. Since $D' \le Z$,
it follows that $D' \le \ker(\lambda)$. By (a), in a minimal counterexample
we can assume that $\lambda$ is faithful. Therefore we conclude that $D$ is
abelian.  By the main result of \cite{KM13}, we conclude that all irreducible
characters in $B$ have height zero.

Assume now that $D \in \Syl_p(G)$ and that all characters in
$\Irr(B|\lambda)$ have height zero. By \cite[Thm.~4.4]{MM}, $\lambda$ extends
to $D$. By \cite[Thm.~6.26 and Cor.~6.27]{Is}, we have that $\lambda$ extends
to $G$. Let $\mu\in\Irr(G)$ be an extension of $\lambda$. Now it is easy to
check that there is a $p$-block $B_1$ of $G$ such that
$\Irr(B_1)=\{\mu^{-1}\chi \mid \chi\in\Irr(B)\}$.
Let $\bar B_1$ be the unique block of $G/Z$ contained in $B_1$. By hypothesis,
we have that all the character in $\bar B_1$ have height zero, and therefore
$D/Z$ is abelian, because Brauer's height zero conjecture is true for
$G/Z$ by assumption.

Assume next that $G$ is $p$-solvable and that all irreducible characters
in $\Irr(B|\lambda)$ have height zero. By the Fong--Reynolds Theorem~9.14 of
\cite{Na} and the Fong Reduction Theorem~10.20 of \cite{Na}, we may assume
that $D \in \Syl_p(G)$. In this case, Conjecture A follows by part~(c) and
the Gluck--Wolf Theorem \cite{GW}.

Finally, assume that $B$ is nilpotent and that all characters in
$\Irr(B|\lambda)$ have height zero. We assume that the reader has some
familiarity with the notation and results in \cite{BP}.
Let $\bar B$ be the unique block of $G/Z$ contained in $B$ (by
\cite[Thm.~9.10]{Na}) and let $\chi\in\Irr(\bar B)$ be of height zero.
We know that $D/Z$ is a defect group of $\bar B$, and therefore $\chi$ has
height zero considered as a character of $B$. By \cite[Thm.~1.2 ]{BP}, the map
$\eta \mapsto \chi * \eta$ defines an isometry between the ring of virtual
characters of $D$ and of $B$. In fact, if $\eta \in \Irr(D)$, then
$\chi*\eta \in \Irr(B)$ and has degree $\chi(1)\eta(1)$ by \cite[page~125]{BP}.
By the definition of the map $\chi*\eta$, and using that $\chi(z)=\chi(1)$,
notice that $(\chi*\eta)(z)=\chi(1)\eta(z)$. Therefore,
$\Irr(B|\lambda)=\{\chi *\eta \mid \eta \in \Irr(D|\lambda)\}$.
By hypothesis, we have that $p$ does not divide $\eta(1)$ for all
$\eta\in\Irr(D|\lambda)$. It follows that $D/Z$ is abelian by Gallagher's
Corollary~6.17 of \cite{Is}. This proves part (e).
\end{proof}

We finish this section by answering Eaton's question.

\begin{thm}   \label{eaton}
 Let $G$ be a finite group, let $N$ be a normal $p$-subgroup of $G$, let
 $\theta\in\Irr(N)$ be $G$-invariant, and let $B$ be a $p$-block of $G$
 with defect group $D$.
 \begin{enumerate}
  \item[\rm(a)] Suppose that some $\chi\in\Irr(B|\theta)$ has the same height
   as $\theta$. Then $\theta$ extends to $D$.
  \item[\rm(b)] Assume that Conjecture~A holds for all finite groups. If all
   $\Irr(B|\theta)$ have the same height as $\theta$ then $D/N$ is abelian.
  \end{enumerate}
\end{thm}

\begin{proof}
By the proof of \cite[Thm.~3.1]{Na1}, there exists a central extension
$\pi: \tilde G \rightarrow G$, with kernel a finite $p$-group
$E \le \CC^{\times}$, such that $\tilde G$ contains $N$ as a normal subgroup
and $\pi^{-1}(N)=\tilde N= N \times E$. Also, $\pi(N)=N$. Furthermore, there
exists $\tau \in \Irr(\tilde G)$ such that $\tau_N=\theta$.
Now, if $\chi \in \Irr(G|\theta)$, $\tilde \chi \in \Irr(\tilde G/E)$
is the corresponding character via the isomorphism $\tilde G/E \rightarrow G$,
then $\tilde\chi$ lies over $\theta$, and by Gallagher's Corollary~6.17 of
\cite{Is}, there exists a unique $\chi^* \in \Irr(\tilde G/N)$ such that
$\tilde\chi=\tau \chi^*$. In fact, if $\lambda\in\Irr(\tilde N/N)$ is given
by $\lambda(n,z)=z^{-1}$, then the map $^*$ defines a character triple
isomorphism $\Irr(G|\theta) \rightarrow \Irr(\tilde G/N |\lambda)$.
We assume that the reader is familiar with the properties of character triple
isomorphisms \cite[Definition~11.23]{Is}.
Now, by \cite[Thm.~9.2]{Na}, notice that $\Irr(B|\theta)$ is not empty.

To prove part (a), assume that $\chi\in\Irr(B|\theta)$ has the same height as
$\theta$. This means that $\chi(1)_p=|G:D|_p\theta(1)$.
Thus $\chi^*(1)_p=|G:D|_p$. Let $B^*$ be the block of $\tilde G/N$ that
contains $\chi^*$, and let $\tilde D/N$ be a defect group of $B^*$. Now if
$\tilde B$ is the block of $\tilde\chi=\tau\chi^*$, then there is a defect
group $P$ of $\tilde B$ such that $\tilde D \subseteq P$ by \cite[Thm.~4.8]{Na}
and \cite[Prop.~2.5(b)]{NS}. Since $\tilde G$ is a central extension of $G$,
we have that $\pi(P)=D$ by \cite[Cor.~2.3(c)]{NR}. Now
\begin{align*}
|G:D|_p&=\chi^*(1)_p=|\tilde G/N:\tilde D/N|_pp^{{\rm ht}(\chi^*)}
  \ge |\tilde G:P|_pp^{{\rm ht}(\chi^*)}\\&
  =|\tilde G/E:P/E|_pp^{{\rm ht}(\chi^*)}=|G:D|_pp^{{\rm ht}(\chi^*)}\, 
\end{align*}
and we deduce that $\tilde D$ is a defect group of $\tilde B$ and that
$\chi^*$ has height zero in $B^*$.
By \cite[Thm.~4.4]{MM}, we have that $\lambda$ extends to $\tilde D/N$.
Now, since the group $D/N$ corresponds to $\tilde D/\tilde N$ under the
character triple isomorphism, we conclude that $\theta$ extends to $D$.
This proves part (a).

Next, we prove (b). Since $\Irr(B|\theta)$ is not empty, fix
$\chi\in\Irr(B|\theta)$, which we know has the height of $\theta$.
By part (a), we know that $\tilde D=\pi^{-1}(D)$ is a defect group of the
block of $\tilde \chi$, and that $\lambda$ extends to $\tilde D/N$.
Now, let $\psi^*\in\Irr(B^*|\lambda)$, where $\psi \in \Irr(G|\theta)$.
Then $\tilde\psi=\psi^*\tau$ and by \cite[Lemma~2.4]{NR} we have that
$\tilde\psi$ and $\tilde\chi$ belong to the same block $\tilde B$ of $\tilde G$.
By \cite[Cor.~2.3]{NR}, it follows that $\chi$ and $\psi$ lie in $B$.  Also, 
by part~(a), we have that $\psi^*$ has height zero. By Conjecture A, we deduce
that $(\tilde D/N)/(\tilde N/N)$ is abelian. This group is isomorphic to $D/N$.
\end{proof}

Notice that the converse of Theorem~\ref{eaton}(b) is false. 
The group $G=\GL_2(3)$ has a unique 2-block, with defect group $P\in\Syl_2(G)$. 
Now, for $N=Q_8$ the unique character $\theta\in\Irr(N)$ of degree~2 extends
to $P$, $P/N$ is abelian, but there are irreducible characters of $G$ over
$\theta$ of degrees~2 and~4.

\section{Quasi-simple groups and Conjecture~A} \label{sec:only if}

Here, we discuss the proof of the following statement:

\begin{thm}   \label{thm:conj A}
 Let $G$ be a finite quasi-simple group. Then no $p$-block of $G$ is a
 minimal counterexample to the ``only if'' direction of the Projective Height
 Zero Conjecture~A.
\end{thm}

Let's start off by making a couple of trivial observations:
\medskip

For the proof of Theorem~\ref{thm:conj A}, by the result of Murai, clearly
we may assume that $D/Z$ is non-abelian. We may further assume that $Z\ne1$,
as otherwise the result (for quasi-simple groups) is contained in
Kessar--Malle \cite{KM17}, and that $\la$ is faithful. Moreover,
by Theorem~\ref{thm:p-solvable}(c) we need not consider the case when $B$ is
of maximal defect, again using \cite{KM17}. This already drastically restricts
the type of situations to consider. We will go through the various cases
according to the Classification of Finite Simple Groups.

\subsection{Sporadic groups and exceptional covering groups}

\begin{prop}   \label{prop:spor}
 Let $G$ be a covering group of a sporadic simple group, or an exceptional
 covering group of a finite simple group of Lie type, or of $\fA_6$ or
 $\fA_7$. Then Conjecture~A holds for $G$.
\end{prop}

\begin{proof}
These are only finitely many groups, with known character tables, and the
assertion can readily be checked with the help of a small GAP programme
\cite{GAP}.
\end{proof}

\subsection{Alternating groups}
Here we deal with the 2-fold coverings $\tilde\fA_n=2.\fA_n$ of the alternating
groups.

\begin{prop}   \label{prop:alt}
 Let $G=2.\fA_n$ be the 2-fold covering group of an alternating group $\fA_n$,
 $n\ge5$. Then Conjecture~A holds for all 2-blocks of $G$.
\end{prop}

\begin{proof}
Let $B$ be a 2-block of $\tilde\fA_n$ and let $D$ denote a defect group of $B$.
Then we need to show that there exists a spin character in $B$ (that is, a
faithful irreducible character of $\tilde\fA_n$) of positive defect whenever
$D/Z$ is non-abelian, where $Z=Z(\tilde\fA_n)$. For this let's first consider
the block $\tilde B$ of the double covering $\tilde\fS_n$ of the symmetric
group $\fS_n$ covering $B$. The
heights of spin characters in $\tilde B$ have been described by Bessenrodt
and Olsson. By \cite[Thm.~1.1 and Cor.~3.10]{BO97} there are spin characters
in $\tilde B$ of height at least~2 unless $\tilde B$ has weight at most~1.
These characters have height at least~1 in $B$, as required. On the other hand,
for blocks of weight at most~1 we have that $|D|\le2$ and so $D$ is abelian.
This achieves the proof.
\end{proof}

\subsection{Groups of Lie type}

There are only finitely many covering groups of simple groups of Lie type
in characteristic~$p$ by a non-trivial $p$-group, and these have all been
considered in Proposition~\ref{prop:spor} already. So when dealing with the
finite groups of Lie type we may now assume that their underlying
characteristic is different from the considered prime, which will be called
$\ell$ here. This is the most difficult case in our analysis.
We consider the usual setup: Let $\bG$ be a simple algebraic group of simply
connected type over an algebraic closure of a finite field, and let
$F:\bG\rightarrow\bG$ be a Frobenius endomorphism with respect to an
$\FF_q$-rational structure such that the finite group of fixed points
$G:=\bG^F$ is quasi-simple. Let $\bG^*$ be in duality with $\bG$ with
corresponding Frobenius map also denoted~$F$ and set $G^*:=\bG^{*F}$. We will
make use of Lusztig's Jordan decomposition of characters. Let $s\in G^*$ be a
semisimple element, and $\cE(G,s)\subseteq\Irr(G)$ the corresponding Lusztig
series. Then Jordan decomposition yields a bijection
$$\Psi:\cE(G,s)\longrightarrow\cE(C_{G^*}(s),1),$$
such that
$$\chi(1)=|G^*:C_{G^*}(s)|_{p'}\,\Psi(\chi)(1)\qquad
  \text{for all }\chi\in\cE(G,s).$$
We will also use the fact, shown by Brou\'e--Michel, that for all semisimple
$\ell'$-elements $s\in G^*$,
$$\cE_\ell(G,s):=\bigcup_{t\in C_{G^*}(s)_\ell}\cE(G,st),$$
where the union runs over $\ell$-elements in $C_{G^*}(s)$, is a union of
$\ell$-blocks.

\begin{lem}   \label{lem:ell-element}
 Let $G$ be as above, let $s\in G^*$ be an $\ell'$-element and $B$ an
 $\ell$-block of $G$ such that $\Irr(B)\subseteq\cE_\ell(G,s)$. Let
 $t\in C_{G^*}(s)$ be an $\ell$-element such that there are
 $\chi\in\cE(G,s)\cap\Irr(B)$ and $\chi'\in\cE(G,st)\cap\Irr(B)$ with
 $|C_{G^*}(s):C_{G^*}(st)|\Psi(\chi')(1)/\Psi(\chi)(1)$ divisible by $\ell$.
 Then $\chi'$ has positive height in $B$.
\end{lem}

\begin{proof}
Our assumptions yield
$$\Big(\frac{\chi'(1)}{\chi(1)}\Big)_\ell
  =\Big(\frac{|G^*:C_{G^*}(st)|\Psi(\chi')(1)}{|G^*:C_{G^*}(s)|\Psi(\chi)(1)}
   \Big)_\ell
  =\Big(\frac{|C_{G^*}(s):C_{G^*}(st)|\Psi(\chi')(1)}{\Psi(\chi)(1)}\Big)_\ell
  >1,$$
so indeed $\chi'$ has positive height.
\end{proof}

\begin{cor}   \label{cor:one block}
 Let $G$ be as above, and let $s\in G^*$ be an $\ell'$-element such that
 $\cE_\ell(G,s)$ is a single $\ell$-block. Then for every $\ell$-element
 $t\in C_{G^*}(s)$ with $|C_{G^*}(s):C_{G^*}(st)|_\ell>1$, the characters in
 $\cE_\ell(G,st)$ have positive height in~$B$.
\end{cor}

\begin{proof}
The Lusztig series $\cE(G,s)$ contains the preimages of the linear characters
of $C_{G^*}(s)$ under Jordan decomposition, the so-called semisimple
characters, of degree $|G^*:C_{G^*}(s)|_{p'}$. If $\chi$ is one of those,
then any $\chi'\in\cE(G,st)$ satisfies the assumptions of
Lemma~\ref{lem:ell-element} and we conclude.
\end{proof}

We will also require the following fact; see e.g.~\cite[3.2]{MN11}:

\begin{prop}   \label{prop:faithful}
 For $G,G^*$ as above there is an isomorphism
 $f:Z(G)\rightarrow G^*/[G^*,G^*]$ such that for any $z\in Z(G)$, the
 characters in $\cE(G,s)$ are non-trivial on $z$ for all $s$ in the coset
 $[G^*,G^*]f(z)$.
\end{prop}

Thus, for example if $Z(G)$ is cyclic of prime order, then the Lusztig series
of any semisimple element $s\in G^*\setminus[G^*,G^*]$ will contain only
faithful characters of $G$.

We first treat groups of type $A$. Here the existence of suitable
$\ell$-elements for the application of Lemma~\ref{lem:ell-element} and
Corollary~\ref{cor:one block} is guaranteed by:

\begin{lem}   \label{lem:GLn}
 Let $G=\GL_m(q)$ and $\ell\le m$ a prime dividing $q-1$. Write $\ell^b$ for
 the precise power of $\ell$ dividing $q-1$. Then there exists an
 $\ell$-element $t\in G$ such that $|G:C_G(t)|_\ell\ge\ell^b$, with strict
 inequality unless $(m,\ell)\ne(2,2),(3,2)$. If $m$ is not a power of $\ell$
 then for any $b'\le b$ we can choose $t$ to have determinant of order
 $\ell^{b'}$.
\end{lem}

\begin{proof}
Let $m=\sum_{i=0}^s a_i\ell^i$ be the $\ell$-adic expansion of $m$ (so
$0\le a_i\le\ell-1$ and $a_s\ne0$). Consider the subgroup
$\prod_{i=0}^s\GL_{\ell^i}(q)^{a_i}$ of $G$ containing a Sylow $\ell$-subgroup
of $G$. Choose $t$ inside a factor $\GL_{\ell^s}(q)$ to be an element of
order $(q^{\ell^s}-1)_\ell$, a power of a Singer cycle. Then the centraliser
of $t$ in $G$ is of the form $\GL_1(q^{\ell^s})\times \GL_{m-\ell^s}(q)$ and
hence has $\ell$-part $(q^{\ell^s}-1)_\ell|\GL_{m-\ell^s}(q)|_\ell$.
\par
Now first assume that $\ell$ is odd or $q\equiv1\pmod4$. Then from the order
formula we get $|\GL_m(q)|_\ell=\ell^{c_m}$ with
$c_m=mb+\sum_ia_i(\ell^i-1)/(\ell-1)$, showing
$|G:C_G(t)|_\ell\ge\ell^{b(\ell^s-1)}\ge\ell^b$, with strict
inequality unless $\ell=2$ and $m\le3$. On the other hand, if $\ell=2$ and
$q\equiv3\pmod4$ then writing $2^d\ge4$ for the precise power of~2 dividing
$q+1$, we have that $|\GL_m(q)|_2=2^{c_m}$ with
$c_m=m+\lfloor\frac{m}{2}\rfloor d+\sum_ia_i(2^i-1)$. Again, the index of
$C_G(t)$ in $G$ has 2-part at least~2, with strict inequality unless $m\le3$.
\par
If $m\ne\ell^s$ then we can adjust $t$ by a suitable $\ell$-element in the
factor $\GL_{m-\ell^s}(q)$ to obtain an $\ell$-element with any prescribed
determinant of $\ell$-power order.
\end{proof}

\begin{prop}   \label{prop:type A}
 The ``only if'' direction of Conjecture~A holds for all $\ell$-blocks of
 $\SL_n(q)$, $\ell|(q-1)$, and of $\SU_n(q)$, $\ell|(q+1)$.
\end{prop}

\begin{proof}
Let $G=\SL_n(q)$ and $\la$ be an irreducible character of $Z(G)$ of
$\ell$-power order. (So in particular $\ell$ divides $\gcd(n,q-1)$). Let
$N\unlhd Z(G)$ with $N_\ell=\ker(\la)$ and consider a block $B$ of $G/N$,
with defect group $D$. We may assume that $D$ is non-abelian. Let $\hat B$ be
the block of $G$ containing $B$, with defect group $\hat D$ such that
$D=\hat D/N_\ell$, and $\tilde B$ a block of $\tilde G=\GL_n(q)$ covering
$\hat B$, with defect group $\tilde D\ge\hat D$. Then by the main result of
\cite{FS82} there exists a semisimple $\ell'$-element $\tilde s\in\GL_n(q)$ such
that $\Irr(\tilde B)=\cE_\ell(\tilde G,\tilde s)$, and $\tilde D$ is conjugate
to a Sylow $\ell$-subgroup of $C_{\tilde G}(\tilde s)$. We then have
that $\Irr(B)\subseteq\Irr(\hat B)\subseteq\cE_\ell(G,s)$, where $s$ is the
image of $\tilde s$ in $G^*=\PGL_n(q)$. Now
$$C_{\tilde G}(\tilde s)\cong
  \GL_{n_1}(q^{e_1})\times\cdots\times\GL_{n_r}(q^{e_r})$$
for suitable $n_j,e_j$ with $\sum_j n_je_j=n$. As $D$ is non-abelian, so is
$\tilde D$, so $n_j\ge\ell$ for at least one $j$. Then by Lemma~\ref{lem:GLn}
there is an $\ell$-element $\tilde t\in C_{\tilde G}(\tilde s)$ with
$|C_{\tilde G}(\tilde s):C_{\tilde G}(\tilde s\tilde t)|_\ell>(q-1)_\ell$,
unless $n_j\le3$ and $\ell=2$. Moreover, unless $r=1$ and $n_1=\ell^k$ for
some $k$, we can arrange for $\tilde t$ to have any prescribed determinant
of $\ell$-power order. Let $t$ be the images of $\tilde t$ in $G^*$. Now
choose the determinant of $\tilde t$ such that $\cE(G,st)$ lies above $\la$.
As $|\tilde G|_\ell/|G^*|_\ell=(q-1)_\ell$ we have that
$|C_{G^*}(s):C_{G^*}(st)|_\ell>1$. Thus, by Corollary~\ref{cor:one block} we
have that $\Irr(B|\la)$ contains characters of positive height, as claimed.
\par
It remains to discuss the exceptions, viz.~either $n_j\le3$ ($1\le j\le r$)
and $\ell=2$, or $r=1$ and $n_1=\ell^k$. First assume that $\ell=2$ and
$n_j\le3$ ($1\le j\le r$). Then again by Lemma~\ref{lem:GLn} we find a
suitable $\ell$-element unless $n_2=\ldots=n_r=1$. If $n_1=3$, or if
$n_1=2$ and $n\ne n_1$ then
$|C_{\tilde G}(\tilde s):C_{\tilde G}(\tilde s\tilde t)|_2
 =|C_{G}(s):C_{G}(s t)|_2$ and we are again done.
\par
So we are only left with the second case, namely that $r=1$ and $n_1=\ell^k$.
Then $C_{\tilde G}(\tilde s)\cong\GL_{\ell^k}(q^e)$ with $\ell^ke=n$. Here, let
$\tilde t$ be an $\ell$-element of $\GL_{\ell^k}(q^e)$ having just one
non-trivial eigenvalue $z$. This can be chosen such that $\tilde t$ has any
prescribed determinant viewed as element of $\GL_n(q)$. Then
$C_{\tilde G}(\tilde s\tilde t)\cong\GL_{\ell^k-1}(q^e)(q^e-1)$ has index
divisible by $\ell^k$ in $C_{\tilde G}(\tilde s)$, and this index remains
unchanged when passing to $G^*=\PGL_n(q)$, unless possibly when $\ell^k=2$,
and we are done as before. So finally assume that $\ell=2$ and
$C_{\tilde G}(\tilde s)\cong\GL_2(q^e)$ with $n=2e$. If $q^e\equiv\pm1\pmod8$
then there are 2-elements in $\PGL_2(q^e)$ of order at least~4 in any
non-trivial coset of $\PSL_2(q)$ not centralised by a Sylow 2-subgroup, and
we may again conclude by Corollary~\ref{cor:one block}. Finally, if
$q^e\equiv\pm3\pmod8$, then $e$ is odd, and $\gcd(n,q-1)=\gcd(2,q^e-1)=2$.
So $|Z|=2$, and $D$ is a Sylow 2-subgroup of $\PGL_2(q^e)$ of order~8. But
then $D/Z$ is abelian.
\par
The case of $G=\SU_n(q)$ is entirely similar.
\end{proof}

For groups of classical type we will need the following:

\begin{lem}   \label{lem:2-syl}
 Let $\bG$ be connected reductive in odd characteristic with a Frobenius map
 $F:\bG\rightarrow\bG$ with respect to an $\FF_q$-rational structure and set
 $G=\bG^F$. If $\bG$ is not a torus, then any coset of $G/[G,G]$ of 2-power
 order contains a 2-element that is not centralised by a Sylow 2-subgroup
 of $G$.
\end{lem}

\begin{proof}
First assume that $q\equiv1\pmod4$. Let $\bT\le\bG$ be a maximally split
torus. Then $N_G(\bT)$ contains a Sylow 2-subgroup $S$ of $G$ (see
e.g.~\cite[Prop.~5.20]{MaH0}), and $N_G(\bT)/T\cong W^F$, where $W$ is the
Weyl group of $\bG$. Let $s\in W^F$ be a fundamental reflection with preimage
$\dot s\in N_G(\bT)$ and $H\le G$ the $A_1$-type subgroup corresponding to $s$.
Then $\dot s$ acts non-trivially on $S\cap T$ and
we thus find elements in $(T\cap S)\setminus Z(S)$ in any coset $G/[G,G]$ of
2-power order. Now $N_G(\bT)$ controls fusion in $T$ \cite[Prop.~5.11]{MaH0},
so no such element is conjugate to an element of $Z(S)$, and hence cannot be
centralised by a Sylow 2-subgroup of $G$.\par
If $q\equiv3\pmod4$, let $\bT\le\bG$ be the centraliser of a $\Phi_2$-torus
of $\bG$. Then we may argue as above, except that $W^F$ has to be replaced
by the reflection group $C_W(w_0F)$ where $w_0\in W$ is the longest element.
\end{proof}

\begin{thm}   \label{thm:lie}
 Let $G$ be a covering group of a simple group of Lie type. Then no
 $\ell$-block of $G$ is a minimal counter-example to the ``only if'' direction
 of Conjecture~A.
\end{thm}

\begin{proof}
As pointed out above, we may assume by Proposition~\ref{prop:spor} that $G$ is
not an exceptional covering group of $G/Z(G)$ and that thus $\ell$ is not the
underlying characteristic of $G$. As we only need to consider the case
when $\ell$ divides $|Z(G)|$, the group $G$ is not a Suzuki or Ree group.
But then we may assume that $G=\bG^F$ as above. Furthermore, $G$ is not of type
$A$ by Proposition~\ref{prop:type A}. Then we have $\ell\le3$.
\par
Let us first assume that $\bG$ is of classical type $B_n$, $C_n$ or $D_n$.
Then $|Z(\bG)|$ is a 2-power, so here $\ell=2$. Let $B$ be a 2-block of $G$.
Then by Enguehard \cite[Prop.~1.5]{En08} there is a semisimple element
$s\in G^*$ of odd
order such that $\Irr(B)=\cE_2(G,s)$ is a Brou\'e--Michel union of Lusztig
series, where $\bG^*$ denotes a group in duality with $\bG$. Furthermore, if
$\bL\le\bG$ denotes a Levi subgroup in duality with $C_{\bG^*}(s)$ (which is
connected as $s$ has odd order) then any Sylow 2-subgroup of $L=\bL^F$ is a
defect group of $B$. By Lemma~\ref{lem:2-syl} any coset of $[G^*,G^*]$ contains
a 2-element $t\in C_{G^*}(s)$ such that $|C_{G^*}(st)|_2<|C_{G^*}(s)|_2$. Then
by Corollary~\ref{cor:one block} the characters in
$\cE(G,st)\subseteq\cE_2(G,s)$ have positive height, and by
Proposition~\ref{prop:faithful} they are faithful on a chosen cyclic subgroup
of the centre $Z(G)$.
\par
Next assume that $G$ is of exceptional type. Then our assumption that
$Z(G)\ne1$ implies that either $G=E_6(q)_\SC$ or $\tw2E_6(q)_\SC$ with
$\ell=3$, or $G=E_7(q)_\SC$ with $\ell=2$. First assume that $G=E_6(q)_\SC$
with $\ell=3$ and $Z(G)\ne1$, so $3|(q-1)$. The unipotent 3-blocks of $G$ and
their non-unipotent constituents are described in
\cite[p.~351 and Thm.~B]{En00}. By Theorem~\ref{thm:p-solvable}(c) we need not
consider the principal 3-block. The only other unipotent 3-block $B$ of $G$ has
defect group $D$ an extension of a homocyclic abelian group $3^a\times 3^a$,
with $3^a$ the precise power of $3$ dividing $q-1$, by a cyclic group of
order~3. Let $t\in G^*$ be an element of order~3 with centraliser of type
$\tw3D_4(q).3$,
not contained in the derived subgroup of $G^*$. Then the cuspidal character
$\chi'$ in the Lusztig series $\cE(G,t)$ is faithful on $Z(G)$, it lies in
$\Irr(B)$ by \cite[Prop.~17]{En00} and it is of positive height by
Lemma~\ref{lem:ell-element}. By \cite[Tab.~3]{KM13} the only further
quasi-isolated 3-block consists of the characters in $\cE_3(G,s)$ for an
involution $s\in G^*$ with centraliser $A_5(q)A_1(q)$. Let $t$ be a 3-element
in $C_{G^*}(s)$ outside $[G^*,G^*]$ with centraliser $A_4(q).A_1(q).\Ph1$,
then we may conclude using Corollary~\ref{cor:one block}. We postpone the
discussion of non-quasi-isolated 3-blocks to the end.
\par
The case of $\tw2E_6(q)_\SC$ with $\ell=3$ and $q\equiv2\pmod3$ is entirely
similar. The relevant data are collected in Table~\ref{tab:exc}.

\begin{table}[ht]
\caption{Characters of positive height}   \label{tab:exc}
$$\begin{array}{cc|cccll}
 G&  \ell&  C_{G^*}(s)& \text{HC}& C_{G^*}(st)& \chi(1)_\ell& \chi'(1)_\ell\\
 \noalign{\hrule}\noalign{\hrule}
 E_6(q)_\SC& 3& G^*& D_4& \tw3D_4(q).\Ph3.3& 3^{4a+3}& 3^{6a+2}\\
 (q\equiv1\ (3))& & A_5(q).A_1(q)& \emptyset& A_4(q).A_1(q).\Ph1& 3^2& 3^3\\
 \noalign{\hrule}
      \tw2E_6(q)_\SC& 3& G^*& D_4& \tw3D_4(q).\Ph6.3& 3^{4a+3}& 3^{6a+2}\\
 (q\equiv2\ (3))& & \tw2A_5(q).A_1(q)& \emptyset& \tw2A_4(q).A_1(q).\Ph2& 3^2& 3^3\\
 \noalign{\hrule}
          E_7(q)_\SC& 2& G^*& E_6[\theta]& E_6(q).\Ph1& 3^{6a+9}& 3^{6a+10}\\
 (q\equiv1\ (4))& & A_5(q).A_2(q)& \emptyset& A_2(q)^3.\Ph1.2& 3^5& 3^6\\
                & & \tw2A_5(q).\tw2A_2(q)& \emptyset& \tw2A_2(q)^3.\Ph2.2& 3^{3a+2}& 3^{4a+2}\\
 \noalign{\hrule}
          E_7(q)_\SC& 2& G^*& \tw2E_6[\theta]& \tw2E_6(q).\Ph2& 3^{6a+9}& 3^{6a+10}\\
 (q\equiv3\ (4))& & A_5(q).A_2(q)& \emptyset& A_2(q)^3.\Ph1.2& 3^{3a+2}& 3^{4a+2}\\
                &  & \tw2A_5(q).\tw2A_2(q)& \emptyset& \tw2A_2(q)^3.\Ph2.2& 3^5& 3^6\\
 \noalign{\hrule}
\end{array}$$
\end{table}

Next assume that $G=E_7(q)_\SC$ with $\ell=2$, so $q$ is odd. First assume that
$q\equiv1\pmod4$. By \cite[p.~354]{En00} there are three unipotent 2-blocks of
$G$ with non-abelian defect, the principal block and two further ones
corresponding to Harish-Chandra series of type $E_6$. Here let $t\in G^*$ be a
2-element with centraliser $E_6(q)(q-1)$ of order $(q-1)_2$. This is not
contained in the derived subgroup $[G^*,G^*]$, and by order comparison it does
not centralise a Sylow 2-subgroup of $G^*$, so by Lemma~\ref{lem:ell-element}
the elements in $\cE(G,t)$ in the $E_6$-Harish-Chandra series provide faithful
characters in the corresponding unipotent blocks of positive height.
By \cite[Tab.~4]{KM13} the only further quasi-isolated 2-blocks of $G$ consist
of the characters in $\cE_2(G,s)$ for elements $s\in G^*$ of order~3 with
centraliser of type $A_5+A_2$. For $q\equiv1\pmod3$, this has rational
structure $A_5(q).A_2(q)$. Let $t$ be a 2-element in $C_{G^*}(s)$ outside
$[G^*,G^*]$ with centraliser $A_2(q)^3.\Ph1$, then we may conclude by
applying Corollary~\ref{cor:one block}. When $q\equiv2\pmod3$, the centraliser
of $s$ has rational structure $\tw2A_5(q).\tw2A_2(q)$, and here we may take
a 2-element $t$ in $C_{G^*}(s)$ with centraliser $\tw2A_2(q)^3.\Ph2$.
The case when $q\equiv3\pmod4$ is entirely similar, again see
Table~\ref{tab:exc}.
\par
Now assume that $B$ is a non-quasi-isolated 3-block of $E_6(q)_\SC$,
$\tw2E_6(q)_\SC$, or a non-quasi-isolated 2-block of $E_7(q)_\SC$. Then
by the main result of Bonnaf\'e--Dat--Rouquier \cite{BDR}, there is a Morita
equivalence  of $B$ to a block $b$ of a proper Levi subgroup $L$ of $G$,
induced by the height preserving bijection $\Irr(B)\rightarrow\Irr(b)$ coming
from Jordan decomposition, and the defect group $D_b$ of $b$ is isomorphic
to~$D$. As we assume that $D/Z$ is non-abelian, by \cite{KM17} there exist
characters of positive height in the block $\bar B$ of $G/Z$ contained in $B$.
But then the block $\bar b$ of $L/Z$ contained in $b$ also has characters of
positive height, whence $D/Z$ is non-abelian by the proven direction \cite{KM13}
of the ordinary height zero conjecture.
Thus, $B$ is certainly not a minimal counter-example to Conjecture~A.
\end{proof}

We have completed the proof of Theorem~\ref{thm:conj A}.



\begin{thebibliography}{131}

\bibitem{BO97}
{\sc C. Bessenrodt, J. B. Olsson}, Heights of spin characters in characteristic
  $2$. \emph{Finite Reductive Groups (Luminy, 1994)},  51--71, Progr. Math.,
  141, Birkh\"auser, Boston, 1997.

\bibitem{BDR}
{\sc C. Bonnaf\'e, J.-F. Dat, R. Rouquier}, Derived categories and
  Deligne--Lusztig varieties II. \emph{Ann. of Math. (2) \bf 185} (2017),
  609--670.

\bibitem{BP}
{\sc M. Brou\'e, L. Puig}, A Frobenius theorem for blocks. \emph{Invent.
  Math. \bf 56} (1980), 117--128.

\bibitem{En00}
{\sc M. Enguehard}, Sur les $l$-blocs unipotents des groupes r\'eductifs
  finis quand $l$ est mauvais. \emph{J. Algebra \bf230} (2000), 334--377.

\bibitem{En08}
{\sc M. Enguehard}, Vers une d\'ecomposition de Jordan des blocs des groupes
  r\'eductifs finis. \emph{J. Algebra \bf319} (2008), 1035--1115.

\bibitem{FS82}
{\sc P. Fong, B. Srinivasan}, The blocks of finite general linear and unitary
  groups. \emph{Invent. Math. \bf69} (1982), 109--153.

\bibitem{GW}
{\sc D. Gluck, T. R. Wolf}, Brauer's height conjecture for $p$-solvable groups.
  \emph{Trans. Amer. Math. Soc. \bf 282} (1984), 137--152.

\bibitem{Is}
{\sc I. M. Isaacs}, \emph{Character Theory of Finite Groups}. AMS Chelsea,
  Providence, 2006.

\bibitem{KM13}
{\sc R. Kessar, G. Malle}, Quasi-isolated blocks and Brauer's height
  zero conjecture. \emph{Ann. of Math. (2) \bf178} (2013), 321--384.

\bibitem{KM17}
{\sc R. Kessar, G. Malle}, Brauer's height zero conjecture for quasi-simple
  groups. \emph{J. Algebra \bf475} (2017), 43--60.

\bibitem{MaH0}
{\sc G. Malle}, Height~0 characters of finite groups of Lie type.
  \emph{Represent. Theory \bf11} (2007), 192--220.

\bibitem{MN11}
{\sc G. Malle, G. Navarro}, Extending characters from Hall subgroups.
  \emph{Doc. Math. \bf16} (2011), 901--919.

\bibitem{MM}
{\sc M. Murai},~~Block induction, normal subgroups and characters
  of height zero. \emph{Osaka J. Math. \bf 31} (1994), 9--25.

\bibitem{Na}
{\sc G. Navarro}, \emph{Characters and Blocks of Finite Groups}. London
  Mathematical Society Lecture Note Series, 250. Cambridge University Press,
  Cambridge, 1998.

\bibitem{Na1}
{\sc G. Navarro}, Brauer characters relative to a normal subgroup. \emph{Proc.
  London Math. Soc. \bf 81} (2000), 55--71.

\bibitem{NS}
{\sc G. Navarro, B. Sp\"ath}, On Brauer's Height Zero Conjecture.
  \emph{J. Eur. Math. Soc. \bf 16} (2014), 695--747.

\bibitem{NR}
{\sc N. Rizo}, $p$-blocks relative to a character of a normal subgroup.
  Preprint.

\bibitem{GAP}
{\sc The GAP~Group}, {\em GAP -- Groups, Algorithms, and Programming,
  Version 4.4}; 2004, \verb+http:+ \verb+ //www.gap-system.org+.

\end{thebibliography}
\end{document}